\documentclass[reqno, 11pt]{amsart}

\usepackage{style}

\title[Removability of product sets for Sobolev functions in the plane]
{Removability of product sets for Sobolev functions in the plane}

\author{Ugo Bindini}
\author{Tapio Rajala}

\address{University of Jyväskylä \\
         Department of Mathematics and Statistics \\
         P.O. Box 35 (MaD) \\
         FI-40014 University of Jyväskylä \\
         Finland}
         
\email{ugo.u.bindini@jyu.fi}         
\email{tapio.m.rajala@jyu.fi}

\thanks{The authors acknowledge the support from the Academy of Finland, grant no. 314789.}
\date{\today}

\begin{document}

\begin{abstract}
We study conditions on closed sets $C,F \subset \R$ making the product $C \times F$ removable or non-removable for $W^{1,p}$. The main results show that the Hausdorff-dimension of the smaller dimensional component $C$ determines a critical exponent above which the product is removable for some positive measure sets $F$, but below which the product is not removable for another collection of positive measure totally disconnected sets $F$. Moreover, if the set $C$ is Ahlfors-regular, the above removability holds for any totally disconnected $F$.
\end{abstract}


\maketitle


\section{Introduction}

In this paper we study the Sobolev-removability of closed subsets of the Euclidean plane. 
The Sobolev space $W^{1,p}(\Omega)$, for $1 \le p \le \infty$ and a domain $\Omega \subset \R^2$, consists of $f \in L^p(\Omega)$ for which the weak first order partial derivatives $\partial_if$ are also in $L^p(\Omega)$.
A subset $E \subset \R^2$ of Lebesgue measure zero is called removable for $W^{1,p}$, or simply $p$-removable, for $1 \le p \le \infty$, if $W^{1,p}(\R^2 \setminus E) = W^{1,p}(\R^2)$ as sets. Since $E$ has Lebesgue measure zero, $E$ is removable for $W^{1,p}$ if and only if every $u \in W^{1,p}(\R^2 \setminus E)$ has an $L^p$-representative that is absolutely continuous on almost every line-segment parallel to the coordinate axis. 

Let us make some observations on $p$-removable sets. By Hölder's inequality, $p$-removable sets are also $q$-removable for every $q>p$. In particular, each $p$-removable set is $\infty$-removable. Hence, the complement of a $p$-removable set is always quasi-convex meaning that any two points in the complement can be joined by a curve in the complement whose length is comparable to the distance between the points. Since the sets $E$ we consider have Lebesgue measure zero, the quasi-convexity of the complement implies that closed $p$-removable sets are actually metrically removable, see \cite[Proposition 3.3]{KalmykovKovalevRajala2019}.

Ahlfors and Beurling \cite{AhlforsBeurling1950} studied removable sets for analytic functions with finite Dirichlet integral (see also the work of Carleson \cite{Carleson1951}). This class of sets coincides with planar $2$-removable sets. Consequently, a lot of work was done on removable sets for quasiconformal maps that are globally homeomorphisms, see for instance \cite{Bishop1994,Gehring1960,HeinonenKoskela1995,Jones1981,Kaufman1984,KaufmanWu1996,Vaisala1969,VodopyanovGoldhstein1975}. Let us point out that Sobolev-removability has also been considered for globally continuous functions, see for example \cite{Ntalampekos2019} and references therein. The version of Sobolev-removability we consider here can be characterized via condenser capacities or extremal distances \cite{AhlforsBeurling1950,AseevSyvev1974,Hedberg1974,Shlyk1990,Vaisala1962,VodopyanovGoldhstein1975,Yamamoto1982}. However, these conditions are not easy to check. Because of this, Koskela \cite{Koskela1999} and Wu \cite{Wu1998} considered Sobolev-removability in terms of different kinds of porosities that are easier to verify. Removability of porous sets for weighted Sobolev spaces \cite{FutamuraMizuta2003} and (weighted) Orlicz-Sobolev spaces \cite{Karak2015,Karak2019} has also been studied. Generalizations of the removability results in the spirit of Ahlfors and Beurling have been done for weighted Sobolev spaces, see for example \cite{DemshinShlyk1995}.

The sets $E$ whose $p$-removability we consider here are of the form $E = C \times F$ where $C,F \subset \R$ are closed. If $C$ or $F$ contains an interval of positive length, it is easy to see that the set $E$ is not $W^{1,p}$-removable for any $1\le p \le \infty$. Therefore, we may assume that both $C$ and $F$ are totally disconnected. Now, on one hand, if both $C$ and $F$ have zero Lebesgue measure, the set $E$ is automatically $W^{1,p}$-removable since almost every line segment parallel to a coordinate axis has empty intersection with $E$. On the other hand, if $C$ and $F$ both have positive Lebesgue measure, the set $E$ has positive Lebesgue measure, hence cannot be removable. We have reduced our study to the following.
\begin{problem}\label{problem}
 Let $1 \le p \le \infty$ and $C,F \subset \R$ be totally disconnected closed subsets with $C$ having zero Lebesgue measure and $F$ positive Lebesgue measure. Under what conditions on $C$ and $F$ is the set $C \times F$ removable for $W^{1,p}$?
\end{problem}

Examples of $p$-removable and non-removable product sets of the type considered in \autoref{problem} have appeared in \cite{Koskela1999}, \cite[Example 2]{Wu1998}, and \cite[Lemma 4.4]{KoskelaRajalaZhang2017}. In \cite{Koskela1999} and \cite{Wu1998} different porosity parameters of sets determined the $p$-removability. In our results the porosity type conditions have only a secondary role and the main parameter is the Hausdorff dimension of the set $C$. Some ideas of the proofs we present here were present in \cite[Lemma 4.4]{KoskelaRajalaZhang2017}, where also the dimension of $C$ was seen to affect the $p$-removability.

Our main results (\autoref{thm:nonremovability} and \autoref{thm:removability}) connect the Hausdorff dimension of $C$ with the $p$-removability of $C \times F$ in the following way. They roughly say that $C\times F$ is not $p$-removable for some $F$, but is $q$-removable for other $F$ when
\[
p < \frac{2-\dim(C)}{1-\dim(C)} < q.
\]
Thus, the dimension of $C$ is sharp for the transition between non-removable and removable examples. However, we emphasize that in our results the positive measure set $F$ needs to be thick (\autoref{thm:nonremovability}) or thin (\autoref{thm:removability}) enough.

Our results should be mainly compared to the $p$-removability results of Koskela \cite{Koskela1999}. He considered the case $C = \{0\}$ and observed in \cite[Theorem 2.2]{Koskela1999} that $ \{0\} \times F$ is not $p$-removable for $1 \le p \le 2$, when $\Haus^1(F)> 0$ and $F = [0,1] \setminus \bigcup_{i=1}^\infty I_i$ with $I_i$ pairwise disjoint open intervals with $\sum_{i=1}^\infty |I_i|^{2-p}$. The generalization of this result is done in \autoref{thm:nonremovability} below. In the other direction, Koskela proved in \cite[Theorem 2.3]{Koskela1999} that $\{0\} \times F$ is $p$-removable for $1 < p < 2$, if for almost every $x \in F$ there exist a sequence of numbers $r_i\searrow 0$ and a constant $c$ so that $B(x,r_i) \setminus F$ contains an interval of length $cr_i^{1/(2-p)}$. We generalize this result in \autoref{thm:removability}.

\begin{theorem}\label{thm:nonremovability}
Let $2 \le p < \infty$ and  $s > \frac{p-2}{p-1}$. Then for any closed subset $C \subset \R$ with $\Haus^s(C) >0$ and any set $F$ of the form $F = [0,1]\setminus \bigcup_{j=1}^\infty I_j$, where $I_j$ are open intervals satisfying
\begin{equation} \label{eq:sumintervals}
\sum_{j=1}^\infty \abs{I_j}^{1-(1-s)(p-1)} < \infty,
\end{equation}
and $\Haus^1(F)>0$, the set $C \times F$ is not $p$-removable. 
\end{theorem}

We do not know what are the sets $C$ in \autoref{thm:nonremovability} for which $C \times F$ is not $p$-removable for every closed set $F \subset \R$ of positive Lebesgue measure. On one hand, if $C$ is a singleton, the $p$-removability depends on $F$ by the results of Koskela \cite{Koskela1999}, as we already discussed above. On the other hand, in \autoref{sec:regular} we show that if $C$ is Ahlfors $s$-regular with $0 < s < 1$, then the $p$-removability is independent of $F$. 

\begin{theorem}\label{thm:removability}
 Let $2 \le p < \infty$ and $s < \frac{p-2}{p-1}$.
 Suppose that $C \subset \R$ is a closed set with $\Haus^s(C) < \infty$ and that $F \subset \R$ is a closed set for which at $\Haus^1$-almost every point $y \in F$
 there exists $r_y > 0$ and $c_y > 0$ so that for any $0 < r < r_y$ we have
\begin{equation}\label{eq:density}
\Haus^1(B(y,r)\setminus F) \ge c_y r^{(1-s)(p-1)}.
\end{equation}
Then the set $C \times F$ is $p$-removable.
\end{theorem}
Notice that $(1-s)(p-1) > 1$ in \autoref{thm:removability} with the choices of $p$ and $s$. Thus, there exist closed sets $F \subset \R$ of positive Lebesgue measure that satisfy \eqref{eq:density} at every point $y \in F$.
One might wonder why in \autoref{thm:removability} we require \eqref{eq:density}
for all small scales $r$ instead of a sequence of scales as in \cite[Theorem A]{Koskela1999}. One reason for our stricter requirement is that in our proof we argue using a sequence of dyadic scales. Even if this could be avoided, the fact that we assume the Hausdorff measure of $C$ to be finite would force us to work on many scales at once. Replacing the Hausdorff measure assumption by box counting dimension assumption might yield the analogous result with a weaker requirement on $F$. 

The proof of \autoref{thm:nonremovability} is inspired by the proof of \cite[Lemma 4.4]{KoskelaRajalaZhang2017} by the second named author together with Koskela and Zhang. In \cite[Lemma 4.4]{KoskelaRajalaZhang2017}, the non-removability was proven for a more regular set $C$, while the removability was done via a curve condition to which we return in \autoref{sec:regular}. The proof of \autoref{thm:nonremovability} is done in \autoref{sec:nonremovability} while
\autoref{thm:removability} is proven in \autoref{sec:removability}. In the final \autoref{sec:regular} we study the relations between $p$-removability, curve conditions, and Ahlfors regularity and lower porosity of $C$. In particular, we show that for Ahlfors-regular $C$ the $p$-removability of $C \times F$ is independent of $F$. The non-removability of $C \times F$ for Ahlfors-regular $C$ might still depend on $F$.

\section{Proof of \autoref{thm:nonremovability}} \label{sec:nonremovability}

In this section we prove \autoref{thm:nonremovability}. The proof is similar to the proof of \cite[Lemma 4.4]{KoskelaRajalaZhang2017}, where a standard Cantor staircase function was extended by hand from horizontal lines passing through $F$ to the whole set $\R^2 \setminus (C \times F)$. This was possible because of the regularity of the Cantor set $C$ that was used.
In the proof of \autoref{thm:nonremovability} we give a more general construction of a suitable Cantor staircase function via Frostman's Lemma, and an extension of the staircase function via averages.

Up to taking a subset of $C$, we can assume that $0 < \Haus^s(C) < \infty$ and that $C$ is compact (say, $C \subset [0,1]$).
For $R > 1$, we will construct a function $u \in W^{1,p}(B(0,R)^2 \setminus (C \times F))$ which is not absolutely continuous on any segment $\gra{y} \times (-R,R)$ for $y \in F$. It follows that $u$ cannot be in $W^{1,p}(B(0,R)^2)$.

By Frostman's Lemma (see for instance \cite[Theorem 8.8]{Mattila1995}), there exists a Borel probability measure $\mu$ concentrated on $C$ satisfying
\begin{equation} \label{eq:frostman} \mu(B(x,r)) \leq c_F r^s \end{equation}
for some constant $c_F > 0$.

We define on $[0,1]$ the non-decreasing function $f(x) = \mu([0,x])$ and we extend it to $B(0,R)$ by letting $f = 0$ on $(-R,0)$ and $f = 1$ on $(1,R)$. Observe that $f$ is not absolutely continuous, since it is constant $\Haus^1$-a.e. on $[0,1]$, but $f(1) - f(0) = 1$.

We extend the function $f$ to $\R \times [0,+\infty)$ by letting
\begin{equation} \label{eq:extension}
    v(x,y) = \frac{1}{2y}\int_{x-y}^{x+y} f(t) \,dt.
\end{equation}

\begin{lemma} \label{lemma:nablav} The extension $v$ defined in \eqref{eq:extension} is differentiable on $B(0,R) \times (0, +\infty)$, and
\[ \nabla v(x,y) = \frac{1}{2y} \pa{f(x+y) - f(x-y), f(x+y) + f(x-y) - 2v(x,y)}. \]

In particular,
\[ \abs{\nabla v(x,y)} \leq \frac{f(x+y) - f(x-y)}{\sqrt{2}y} = \frac{\mu(B(x,y))}{\sqrt{2}y}. \]
\end{lemma}

\begin{proof}
By the Leibniz integral rule we have
\[ \frac{dv}{dx}(x,y) = \frac{1}{2y} \frac{d}{dx} \int_{x-y}^{x+y} f(t) \,dt = \frac{f(x+y) - f(x-y)}{2y}  \]
and
\[
\begin{split}
    \frac{dv}{dy}(x,y) &= \frac{-1}{2y^2} \int_{x-y}^{x+y} f(t) \,dt + \frac{1}{2y} \frac{d}{dy} \int_{x-y}^{x+y} f(t) \,dt \\
    & = \frac{1}{2y} \pa{-2 v(x,y) + f(x+y) + f(x-y)}.
\end{split}
\]

The final estimate comes from the fact that $f$ is non-decreasing, which implies
\[ v(x,y) = \frac{1}{2y} \int_{x-y}^{x+y} f(t) \,dt \geq f(x-y). \qedhere \]
\end{proof}

It will be useful to estimate the integral of $\abs{\nabla v}$ on a rectangle $(-R,R) \times (0,r)$. For $0 < y < r$, let $\gra{B(x_i,r_i)}_i$ be a finite cover of $C$ with disjoint balls of radii $r_i < y$. Then we have
\begin{align*}
    \int_{-R}^R \mu(B(x,y)) \,dx &= \int_{-R}^R \sum_i \mu(B(x,y) \cap B(x_i,r_i)) \,dx \leq \sum_i \int_{x_i-2y}^{x_i+2y} \mu(B(x,y) \cap B(x_i,r_i)) \,dx \\
    &\leq 4y \sum_i  \mu(B(x_i,r_i)) = 4y,
\end{align*}
where we used that $\mu$ is a probability measure on $C$. Combining this with \autoref{lemma:nablav} and \eqref{eq:frostman} yields
\begin{equation} \label{eq:stripestimate}
\begin{split}
    \int_0^r \int_{-R}^R \abs{\nabla v}^p \,dx \,dy &\leq 2^{-\frac{p}{2}} \int_0^r \int_{-R}^R \frac{\mu(B(x,y))^p}{y^p} \,dx \,dy  \\
    &\leq 2^{-\frac{p}{2}} c_F \int_0^r \frac{y^{s(p-1)}}{y^p} \int_{-R}^R \mu(B(x,y)) \,dx \,dy \\
    &\leq 2^{2-\frac{p}{2}}c_F \int_0^r y^{(s-1)(p-1)} \,dy.
\end{split}
\end{equation}

We now define the function $u$ as $u(x,y) = v(x,\dist(y,F))$. Observe that $u(x,y) = f(x)$ for every $y \in F$, so $u$ is not absolutely continuous on every segment $\gra{y} \times (-R,R)$, $y \in F$.

Since by hypothesis $(s-1)(p-1) > -1$, making use of \eqref{eq:stripestimate}, for each interval $I_j$ in the complement of $F$ we have
\[ \int_{I_j} \int_{-R}^{R} \abs{\nabla u}^p \,dx \,dy = 2 \int_0^{\abs{I_j}/2} \int_{-R}^{R} \abs{\nabla v}^p \,dx \,dy \leq c(p,s) c_F \abs{I_j}^{1 - (1-s)(p-1)}, \]
where $c(p,s) = 2^{1 + \frac{p}{2} - s(p-1)}$. By summing over $j$ and using \eqref{eq:sumintervals} we obtain that $u \in W^{1,p}(B(0,R)^2 \setminus (C \times F))$, as wanted.

\section{Proof of \autoref{thm:removability}} \label{sec:removability}


Let $u \in W^{1,p}(\R^2 \setminus E)$. We aim at showing that $u \in W^{1,p}(\R^2)$, which holds exactly when $u$ has an $L^p$-representative that is ACL in $\R^2$.
Without changing the notation, let $u$ be the ACL (in $\R^2 \setminus E$) representative of $u$.
Since $\Haus^1(C) =0$, $u$ is absolutely continuous on almost every vertical line-segment in $\R^2$. Hence, we only need to verify that $u$ is absolutely continuous on almost every horizontal line-segment.

Let us write $\alpha=(1-s)(p-1)$.
Let $y \in F$ be such that there exist $c_y > 0$ and $r_y>0$ so that for any $0 < r < r_y$ we have
\begin{equation}\label{eq:poro}
\Haus^1(B(y,r) \setminus F) \ge c_y r^\alpha.
\end{equation}
By assumption, such constants exist for almost every $y$. Let us abbreviate $f(x) = u(x,y)$. It remains to show that $f$ is absolutely continuous.

Let $0 < \delta < r_y$. Recalling that $\Haus^s(C) < \infty$, we take a collection of open intervals $\gra{J_i}_{i = 1}^n$ such that $C \subset \bigcup_{i=1}^n J_i$, $\abs{J_i} < \delta$ for all $i$ and
\begin{equation}\label{eq:Hausdorffbound}
 \sum_{i=1}^n \abs{J_i}^s \le 2\Haus^s(C).
\end{equation}
Without loss of generality, we may assume that no point in $\R$ is contained in more than two different intervals $J_i$.
Define for every $i$ the open square
\[
Q_i = J_i \times B\pa{y, \frac12\abs{J_i}}.
\]


\begin{lemma} \label{lemma:oscillation}
    For every $i$ we have the inequality
    \begin{equation}\label{eq:onecubeestimate}
        \abs{\inf_{x \in J_i}f(x) - \sup_{x \in J_i}f(x)} \le c(s,p,y) \abs{J_i}^{\frac{s}{q}}\|\nabla u\|_{L^p(Q_i)},
    \end{equation}
    where $\frac1p + \frac1q = 1$ and $c(s,p,y) > 0$ is a constant depending only on $s$, $p$, and $y$.
\end{lemma}

Assuming for the moment \autoref{lemma:oscillation}, we conclude the proof as follows. By Hölder's inequality, \eqref{eq:Hausdorffbound}, and \eqref{eq:onecubeestimate} we obtain
\begin{align*}
\sum_{i=1}^n \abs{\inf_{x \in J_i}f(x) - \sup_{x \in J_i}f(x)} & \le \pa{\sum_{i=1}^n\abs{J_i}^{s}}^\frac1q
\pa{\sum_{i=1}^n\abs{J_i}^{-\frac{sp}{q}}\abs{\inf_{x \in J_i}f(x) - \sup_{x \in J_i}f(x)}^p}^\frac1p \\
& \le 
\pa{2\Haus^s(C)}^\frac1q
\pa{\sum_{i=1}^n c(s,p,y) \norm{\nabla u}_{L^p(Q_i)}^p}^\frac1p \\
& \le c(s,p,y)
\pa{2\Haus^s(C)}^\frac1q \norm{\nabla u}_{L^p(\R\times[y-\delta,y+\delta])}\longrightarrow 0
\end{align*}
as $\delta \to 0$.
Since $f$ is absolutely continuous outside $C$, the above shows that $f$ is absolutely continuous on the whole $\R$.

It remains to prove \autoref{lemma:oscillation}.

\begin{proof}[Proof of \autoref{lemma:oscillation}]
Fix $i \in \gra{1,\dotsc, n}$ and let $I \subset J \subset J_i$ be intervals such that $\abs{J} = 2\abs{I}$.
By \eqref{eq:poro}, the open set $K = B(y,\abs{I})\setminus F$ satisfies
$\Haus^1(K) \ge c_y \abs{I}^\alpha$.
Define a collection $\gamma_t\colon [0,1] \to Q_i$, $t \in [0,1]$ of curves so that each $\gamma_t$ is the concatenation of three line-segments $\gamma_t^1$, $\gamma_t^2$, and $\gamma_t^3$ that are defined as follows.
Write $J = [a,b]$, $I=[c,d]$ and set $x_1(t) = a+t(b-a)$ and $x_2(t) = c + t(d-c)$. Define 
\[
y(t) = \inf\left\{\tilde y \in [y-|I|,y+|I|]\,:\,\Haus^1(K\cap(-\infty,\tilde y]) \ge t\Haus^1(K)\right\}.
\]
Now, $\gamma_t^1$ is taken to be the line-segment from $(x_1(t),y)$ to $(x_1(t),y(t))$, $\gamma_t^2$ the line-segment from $(x_1(t),y(t))$ to $(x_2(t),y(t))$, and $\gamma_t^3$ the line-segment from $(x_2(t),y(t))$ to $(x_2(t),y)$. Notice that the image of $\gamma_t$ does not intersect $E$ for $\Haus^1$-almost every $t \in [0,1]$.

By integrating over the curves $\gamma_t$ we obtain
\[
\begin{split}
\abs{\frac{1}{\abs{I}}\int_I f(x)\,dx - \frac{1}{\abs{J}}\int_J f(x)\,dx}
& = \abs{\int_0^1 u(\gamma_t(1)) - u(\gamma_t(0))\,dt}\\
& \le \int_0^1 \abs{u(\gamma_t(1)) - u(\gamma_t(0))}\,dt\\
& \le \int_0^1 \int_{\gamma_t}|\nabla u(z)|\,ds(z)\,dt\\
& = \sum_{k = 1}^3 \int_0^1 \int_{\gamma^k_t}|\nabla u(z)|\,ds(z)\,dt
\end{split}
\]


First we treat the integrals along the vertical lines $\gamma^1_t, \gamma^3_t$. By Hölder's inequality we have
\[
\begin{split}
\int_0^1 \int_{\gamma^1_t}|\nabla u(z)|\,ds(z)\,dt & \leq \int_0^1 \int_{y-\frac{1}{2}\abs{J_i}}^{y+\frac{1}{2}\abs{J_i}} \abs{\nabla u(x_1(t), z)}\, dz \,dt \\
& = \int_I \int_{y-\frac{1}{2}\abs{J_i}}^{y+\frac{1}{2}\abs{J_i}} \abs{\nabla u(x, z)} \,dz \,dx \\
&\leq (\abs{I} \cdot \abs{J})^{\frac1q} \norm{\nabla u}_{L^p(Q_i)} \\
& = c(p) \delta^{\frac{2-s}{q}} \abs{J}^{\frac{s}{q}} \norm{\nabla u}_{L^p(Q_i)}.
\end{split}
\]

A similar computation shows that $\int_0^1 \int_{\gamma^3_t}|\nabla u(z)|\,ds(z)\,dt \leq c(p) \delta^{\frac{2-s}{q}} \abs{J}^{\frac{s}{q}} \norm{\nabla u}_{L^p(Q_i)}$.

To evaluate the integrals along $\gamma^2_t$, observe that the map $t \mapsto y(t)$ is piecewise affine (on a countable union of open intervals) with $y'(t) = \frac{1}{\Haus^1(K)}$ a.e. on $(0,1)$. Thus we have
\[
\begin{split}
    \int_0^1 \int_{\gamma^1_t} \abs{\nabla u(z)} \,ds(z) \,dt &\leq \int_{J \times K} \Haus^1(K)^{-1} \abs{\nabla u(w,z)} \,dw \,dz \\
    &\leq \abs{J}^{\frac1q} \Haus^1(K)^{\frac{1}{q}-1} \norm{\nabla u}_{L^p(J\times K)} \\
    &\leq 2^{\frac{\alpha}{p}} c_y^{-\frac{1}{p}} \abs{J}^{\frac{1}{q} - \frac{\alpha}{p}} \norm{\nabla u}_{L^p(Q_i)} = 2^{\frac{\alpha}{p}} c_y^{-\frac{1}{p}} \abs{J}^{\frac{s}{q}} \norm{\nabla u}_{L^p(Q_i)},
\end{split}
\]
where we used $\Haus^1(K) \geq 2^{-\alpha} c_y \abs{J}^\alpha$ and the definition of $\alpha$.

By putting all together we get
\begin{equation} \label{eq:onescale}
    \abs{\frac{1}{\abs{I}}\int_I f(x)\,dx - \frac{1}{\abs{J}}\int_J f(x)\,dx} \leq c(s,p,y) \abs{J}^{\frac{s}{q}} \norm{\nabla u}_{L^p(Q_i)}.
\end{equation}

Now, let $z_1,z_2 \in J_i$ be such that
\[
\abs{\inf_{x \in J_i}f(x) - \sup_{x \in J_i}f(x)} \le 2\abs{f(z_1)-f(z_2)}.
\]
Let $\gra{I_k}_{k=1}^\infty$ be subintervals of $J_i$ so that $I_1 = J_i$, $z_1 \in I_k$ for every $k \in \N$, and $\abs{I_k} = 2\abs{I_{k+1}}$ for every $k \in \N$. Now, by \eqref{eq:onescale}, we have
\begin{align*}
\abs{f(z_1) - \frac{1}{\abs{J_i}}\int_{J_i}f(x)\,dx}
&\le \sum_{k=1}^\infty \abs{\frac{1}{\abs{I_k}}\int_{I_k} f(x)\,d x - \frac{1}{\abs{I_{k+1}}}\int_{I_{k+1}} f(x)\,d x}\\
& \le \sum_{k=1}^\infty c(s,p,y) \abs{I_k}^{\frac{s}{q}}\|\nabla u\|_{L^p(Q_i)}\\
& \le c(s,p,y) \abs{J_i}^{\frac{s}{q}}\|\nabla u\|_{L^p(Q_i)}.
\end{align*}
Together with an analogous estimate for $z_2$, we obtain
\begin{align*}
\frac12\abs{\inf_{x \in J_i}f(x) - \sup_{x \in J_i}f(x)} & \le \abs{f(z_1)-f(z_2)}\\
& \le \abs{f(z_1) - \frac{1}{\abs{J_i}}\int_{J_i}f(x)\,dx} + \abs{f(z_2) - \frac{1}{\abs{J_i}}\int_{J_i}f(x)\,dx}\\
& \le 2c(s,p,y) \abs{J_i}^{\frac{s}{q}}\|\nabla u\|_{L^p(Q_i)}. \qedhere
\end{align*}
\end{proof}

\section{Curve-condition, porosity and Ahlfors-regular sets}\label{sec:regular}

In this section we study the case where the set $E = C \times F$ consists of a set $F$ of positive measure and a zero measure set $C$ with more regularity. The most regular case is when $C$ is (Ahlfors) $s$-regular, that is, if there exists a constant $c_R>0$ so that
\[
 \frac{1}{c_R}r^s \le \mathcal H^s(B(x,r)\cap C) \le c_Rr^s
\]
for every $x \in C$ and $0 < r < \diam(C)$.
The set $C$ in \cite[Lemma 4.4]{KoskelaRajalaZhang2017} was not exactly $s$-regular, but almost. A small perturbation to $s$-regularity was required there to have the nonremovability at the critical exponent.

In \cite[Lemma 4.4]{KoskelaRajalaZhang2017}, the $p$-removability of $C \times F$ was proven via the following sufficient condition from \cite{Koskela1998,Shvartsman2010}. Suppose $E \subset \mathbb R^2$ is closed set of measure zero and $2\le p < \infty$. If there exists a constant $c_\Gamma>0$ such that for every $z_1,z_2 \in \R^2\setminus E$ there exists a curve $\gamma \subset \R^2 \setminus E$ connecting $z_1$ to $z_2$ and satisfying
\begin{equation}\label{eq:curvecondition}
    \int_\gamma \dist(z,E)^{\frac1{1-p}}\,ds(z) \le c_{\Gamma} \abs{z_1-z_2}^{\frac{p-2}{p-1}},
\end{equation}
then $E$ is $p$-removable. If the above holds, we say that $E \subset \R^2$ satisfies the curve condition \eqref{eq:curvecondition}.

By adapting the proof in \cite{KoskelaRajalaZhang2017}, we get a $p$-removability result that is independent of the structure of $F$.

\begin{theorem}\label{thm:regular}
 Let $C \subset \R$ be a closed $s$-regular set with $0 < s < 1$, and $F \subset \R$ totally disconnected closed set. Then $C \times F$ is $p$-removable for every $p > \frac{2-s}{1-s}$.
\end{theorem}

A slightly more general result for $p$-removability via the curve condition \eqref{eq:curvecondition} than the one stated in Theorem \ref{thm:regular} is in terms of porosity. Recall that a set $C \subset \R$ is called uniformly lower $\alpha$-porous, if for every $x \in C$ and $r>0$ there exists $y \in B(x,r)$ so that $B(y,\alpha r) \cap C = \emptyset$.

\begin{theorem}\label{thm:porous}
 Let $C \subset \R$ be a closed uniformly lower $\alpha$-porous set and $F \subset \R$ totally disconnected closed set. Then $C \times F$ is $p$-removable for every $p > \hat p$, where $\hat p > 2$ depends only on the parameter $\alpha$.
\end{theorem}

\begin{proof}[Proof of Theorems \ref{thm:regular} and \ref{thm:porous}]
Both of the theorems are proven by verifying the condition \eqref{eq:curvecondition}. Towards verifying this condition, let $z_1,z_2 \in \R^2 \setminus E$. Write these points in coordinates as $z_1=(x_1,y_1)$ and $z_2=(x_2,y_2)$. Let us abbreviate $r = \abs{z_1-z_2}$. Since $F$ is totally disconnected and $E$ is closed, we may assume that $y_1,y_2 \notin F$.

Notice that an $s$-regular set is uniformly lower porous. Thus, in both cases by porosity of $C$ there exists a point $x \in B(x_1,r)$ so that $B(x,\alpha r) \cap C = \emptyset$. We now connect $z_1$ to $z_2$ by concatenating three line-segments $\gamma_1$, $\gamma_2$, and $\gamma_3$. The curve $\gamma_1$ connects $(x_1,y_1)$ to $(x,y_1)$, $\gamma_2$ connects $(x,y_1)$ to $(x,y_2)$, and $\gamma_3$ connects $(x,y_2)$ to $(x_2,y_2)$. The choice of $x$ now gives
\[
 \int_{\gamma_2} \dist(z,E)^{\frac1{1-p}}\,ds(z) \le
 \int_{\gamma_2} (\alpha r)^{\frac1{1-p}}\,ds(z) = (\alpha r)^{\frac1{1-p}} \abs{y_1-y_2} \le
 \alpha^{\frac{1}{1-p}}r^{\frac{p-2}{p-1}}
\]
for the vertical part $\gamma_2$. 

For the horizontal parts $\gamma_1$ and $\gamma_3$ we first show that the following condition holds for $s$-regular sets $C$ and for uniformly lower $\alpha$-porous sets $C$ with some $0 < s < 1$: 
there exists a constant $c_s < \infty$ such that for all $0 < \delta \le 1$ and every $-\infty < a < b < \infty$, the set $(a,b) \setminus C$ contains at most $c_s\delta^{-s}$ connected components of length more than $\delta\abs{b-a}$.

Let us first show this for an $s$-regular set $C$. Suppose that $\{I_i\}_{i=1}^n$ are the connected components of $(a,b) \setminus C$ of length more than $\delta\abs{b-a}$. For each $i$ let $v_i$ be the left-most point of $\overline{I}_i$. The sets $(B(v_i,\delta\abs{b-a}) \cap C) \subset [a-\abs{b-a},b+\abs{b-a}]$ are pairwise disjoint.
Thus, by $s$-regularity (notice that the left-most $v_i$ might not be in $C$)
\[
\frac{n-1}{c_R}\left(\delta\abs{b-a}\right)^s
\le 
\Haus^s([a-\abs{b-a},b+\abs{b-a}]\cap C) < c_R(2\abs{b-a})^s,
\]
which gives the claim for $s$-regular sets $C$.

Let us now suppose that $C$ is uniformly lower $\alpha$-porous, fix $\delta$ and denote by $\gra{I_i}_{i = 1}^n$ the intervals of $(a,b) \setminus C$ of length at least $\delta\abs{b-a}$, and by $\{J_i\}_{i = 1}^\infty$ the remaining intervals of $(a,b) \setminus C$. Consider the set $C' = C + B\pa{0,\frac{\delta}{2}\abs{b-a}}$. By a result of A. Salli \cite[Theorem 3.5]{Salli1991}, we have
\begin{equation} \label{eq:salli}
\Haus^1(C') \leq c(\alpha) \abs{b-a} \delta^{1-s},
\end{equation}
where $s = \frac{\log 2}{\log \pa{\frac{2-\alpha}{1-\alpha}}} \in (0,1)$ and $c(\alpha)$ is a positive constant depending on $\alpha$. Observe that $\bigcup_{i} J_i \subset C'$ and, for every interval $I_i$, $\abs{I_i \setminus C'} \leq \abs{I_i} - \frac{\delta}{2}\abs{b-a}.$ Thus, using \eqref{eq:salli}, we have
\[ \abs{b-a} = \sum_{i = 1}^n \abs{I_i} + \sum_{i = 1}^\infty \abs{J_i} = \sum_{i = 1}^n \abs{I_i \setminus C'} + \Haus^1(C') \leq \abs{b-a} - \frac{1}{2}n\delta\abs{b-a} + c(\alpha) \abs{b-a} \delta^{1-s}, \]
yielding $n \leq 2c(\alpha) \delta^{-s}$.


Let us then estimate the integral along $\gamma_1$.
Without loss of generality we may assume that $x_1 < x$. Denote by $\{J_i\}_i$ the collection of open intervals constituting the connected components of $(x_1,x) \setminus C$. Let $k_0 \in \Z$ be so that 
$2^{-k_0}< \abs{x-x_1} \le 2^{-k_0+1}$. Then
\begin{align*}
\int_{\gamma_1} \dist(z,E)^{\frac1{1-p}} \,ds(z)
& \le \sum_{i}2\int_0^{\abs{J_i}}t^{\frac1{1-p}} \,dt
= 2\frac{p-1}{p-2}\sum_{i}\abs{J_i}^{\frac{p-2}{p-1}}\\
& \le c(p)\sum_{k = k_0}^\infty \# \gra{i \st 2^{-k-1} < \abs{J_i} \le 2^{-k}} 2^{-k\frac{p-2}{p-1}} \\
& \le c(p) \sum_{k = k_0}^\infty c_s 2^{(k-k_0)s} 2^{-k\frac{p-2}{p-1}} \\
& \leq c(p) \sum_{k = k_0}^\infty 
c_s 2^{(k-k_0)\pa{s-\frac{p-2}{p-1}}} \abs{x-x_1}^\frac{p-2}{p-1}\\
& \le c(p,s) \abs{x-x_1}^\frac{p-2}{p-1}
\le c(p,s) \abs{z_1-z_2}^\frac{p-2}{p-1}
\end{align*}
as long as $s < \frac{p-2}{p-1}$.

 The integral along $\gamma_3$ is handled analogously.
\end{proof}

We end this section by showing that the $p$-removability results that are proven via the curve condition \eqref{eq:curvecondition} give removability only for porous sets.

\begin{proposition}\label{prop:curvetoporous}
 Suppose that $E = C \times F \subset \R^2$ is a compact set satisfying the curve condition \eqref{eq:curvecondition} and that $F\subset \R$ is a totally disconnected set with positive Lebesgue measure. Then $C$ is uniformly lower $\alpha$-porous for some $\alpha >0$.
\end{proposition}
\begin{proof}
Let $c_\Gamma > 0$ be the constant in \eqref{eq:curvecondition}.
Let $y \in F$ be a Lebesgue density-point of $F$ and $\varepsilon \eqdef \sqrt{2}c_\Gamma^{1-p}$. Then there exists $r_0>0$ such that 
for all $0 < r < r_0$ we have
\begin{equation}\label{eq:Fdensity}
 \Haus^1(B(y,r)\setminus F) < \varepsilon r.
\end{equation}
Let $x \in C$ and $0 < r < r_0$. Define
$\tilde z_1 = (x-r/2,y)$ and $\tilde z_1 = (x+r/2,y)$, and select points $z_1 \in B(\tilde z_1,r/4) \setminus E$ and $z_2 \in B(\tilde z_2,r/4) \setminus E$. Let $\gamma \subset \R^2 \setminus E$ be a curve connecting $z_1$ to $z_2$ and satisfying \eqref{eq:curvecondition}.
Define $A \eqdef \overline{B(x,7r/8)} \times \overline{B(y,r/2)}$
and $d \eqdef \max \gra{ \dist(z,E) \st z \in \gamma \cap A}$. Now, by \eqref{eq:curvecondition}
\[
d^\frac{1}{1-p}\frac{r}{2} \le \int_{\gamma \cap A}  \dist(z,E)^{\frac{1}{1-p}} \,ds(z) \le c_{\Gamma} \abs{z_1-z_2}^{\frac{p-2}{p-1}} \le c_{\Gamma} (2r)^{\frac{p-2}{p-1}}.
\]
Thus,
\[
d \ge 2c_\Gamma^{1-p}r = \sqrt{2}\varepsilon r,
\]
which together with \eqref{eq:Fdensity} gives the $\varepsilon$-porosity of $C$ at $x$ at the scale $r$. From the compactness of $C$ it then follows that $C$ is uniformly lower $\alpha$-porous for some $\alpha>0$.
\end{proof}

\bibliographystyle{plain}
\bibliography{biblio.bib}

\end{document}